\documentclass[11pt,letterpaper,reqno]{amsart}
\usepackage{tikz}
\usetikzlibrary{positioning, shapes.geometric, arrows.meta, calc}
\usepackage{amssymb}
\usepackage{amsmath}
\usepackage{amsthm}
\usepackage{amsfonts}
\usepackage{bbm}
\usepackage{enumitem}
\usepackage{pgfplots}
\pgfplotsset{compat=1.18}
\usepackage{booktabs}
\usepackage{graphicx}
\usepackage[T1]{fontenc}
\usepackage{doi}
\usepackage{float}
\usepackage{bookmark}
\usepackage{hyperref}
\hypersetup{pdfstartview={FitH}}

\newtheorem{theorem}{Theorem}[section]
\newtheorem{lemma}[theorem]{Lemma}
\newtheorem{proposition}[theorem]{Proposition}
\newtheorem{corollary}[theorem]{Corollary}

\newtheorem{problem}[theorem]{Problem}

\theoremstyle{definition}

\newtheorem{remark}[theorem]{Remark}

\numberwithin{equation}{section}

\newcommand{\D}{\mathbb D}
\newcommand{\Pclass}{\mathcal P}
\newcommand{\Log}{\operatorname{Log}}
\newcommand{\C}{\mathbb{C}}

\begin{document}

\title[Optimal universal growth]{Optimal universal growth for integral means of normalized logarithmic derivatives in the Carath\'eodory class}

\author[Y.~He]{Yixin He}
\address{School of Mathematical Sciences, Fudan University, Shanghai 200433, P.~R.~China}
\email{hyx717math@163.com}

\author[Q.~Tang]{Quanyu Tang}
\address{School of Mathematics and Statistics, Xi'an Jiaotong University, Xi'an 710049, P.~R.~China}
\email{tang\_quanyu@163.com}

\author[T.~Zhang]{Teng Zhang}
\address{School of Mathematics and Statistics, Xi'an Jiaotong University, Xi'an 710049, P.~R.~China}
\email{teng.zhang@stu.xjtu.edu.cn}




\begin{abstract}
We determine the optimal universal growth scale for the integral means of normalized logarithmic derivatives in the Carath\'eodory class. This resolves a problem of D.~K.~Thomas.
\end{abstract}

\maketitle

\section{Introduction}\label{sec:intro}

Throughout the paper, \(\D:=\{z\in\C:|z|<1\}\) denotes the open unit disc, and
\(\overline{\D}:=\{z\in\C:|z|\le 1\}\) its closure. We write \(H(\D)\) for the
space of holomorphic functions on \(\D\), and
\[
\Pclass:=\{p\in H(\D): \Re p(z)>0 \text{ for all } z\in\D\}
\]
for the non-normalized Carath\'eodory class. In particular, we do not impose the
normalization \(p(0)=1\). Since all extremal constructions in this paper in fact
satisfy \(p(0)=1\), this convention has no effect on the substance of our results.
We also write \(A(r)\asymp B(r)\) as \(r\to1^-\) if there exist constants
\(c,C>0\) such that
\[
c\,B(r)\le A(r)\le C\,B(r)
\]
for all \(r\) sufficiently close to \(1\).

The Carath\'eodory class \(\Pclass\) is a classical and central object in geometric function theory. By the Herglotz representation, each \(p\in\Pclass\) may be written in the form
\[
p(z)= i\,\operatorname{Im} p(0)+\int_{\partial\D}\frac{\zeta+z}{\zeta-z}\,d\mu(\zeta),
\qquad z\in\D,
\]
for a finite positive Borel measure \(\mu\) on \(\partial\D\); see, for example, \cite[p.~21, Theorem~3.2]{Sim19}. Equivalently, after a Cayley transform, \(\Pclass\) corresponds to the Schur class of analytic self-maps of \(\D\). Thus functions with positive real part lie at the intersection of the analytic, measure-theoretic, and Hardy-space aspects of the subject, and they also provide the Herglotz data in Loewner--Kufarev theory; see, for example, \cite{Dur83,Hay61,HT80,Sim19,TTV18}.

The importance of \(\Pclass\) for univalent function theory is especially transparent
through the standard differential characterizations of geometric subclasses. A normalized
analytic function \(f\) is \emph{starlike} if and only if
\[
\operatorname{Re}\frac{z f'(z)}{f(z)}>0
\qquad (z\in\D),
\]
and convex if and only if
\[
\operatorname{Re}\left(1+\frac{z f''(z)}{f'(z)}\right)>0
\qquad (z\in\D).
\]
Accordingly, the Carath\'eodory class is not merely an auxiliary family: it parametrizes
two of the most basic subclasses of analytic functions on \(\D\), namely starlike and convex functions, and
therefore governs many coefficient, growth, and boundary questions in geometric
function theory \cite{Dur83,TTV18}.

Against this background, it is natural to consider the \emph{normalized logarithmic derivative} \cite{Gri91}
\[
z\frac{p'(z)}{p(z)}.
\]
There are several reasons for this. First, \(z p'(z)/p(z)\) is precisely the quantity that
characterizes starlikeness. Second, on the circle \(z=re^{i\theta}\) one has
\[
\frac{d}{d\theta}\Log p(re^{i\theta})
= i\,z\frac{p'(z)}{p(z)},
\]
so the integral means of \(z p'(z)/p(z)\) measure the angular oscillation of \(\Log p\) on
\(|z|=r\). Third, if
\[
\Log p(z)=a_0+\sum_{n=1}^{\infty}a_n z^n,
\]
then
\[
z\frac{p'(z)}{p(z)}=\sum_{n=1}^{\infty} n a_n z^n,
\]
and Parseval's identity converts its \(L^2\)-integral means into the weighted square sum
\(\sum n^2|a_n|^2 r^{2n}\). This makes \(z p'(z)/p(z)\) the natural normalized quantity for
the Hardy-space estimates carried out below.

The choice of \(z p'(z)/p(z)\) is  also historically natural in view of Thomas's work. Already in the late 1960s and 1970s, starlike 
functions were studied through area theorems, order questions, and asymptotic integral formulas.
In particular, Holland and Thomas \cite{HT69} proved an area theorem for starlike  functions,
Holland and Thomas \cite{HT71} investigated the order of a starlike  function, and London and
Thomas \cite{LT76} studied integrals of the form
\[
H(r)=\int_0^{2\pi}|f(re^{i\theta})|^\sigma
|F(re^{i\theta})|^\tau
\bigl(\operatorname{Re}F(re^{i\theta})\bigr)^\kappa\,d\theta,
\quad
F(z)=\frac{z f'(z)}{f(z)},\quad \sigma,\tau,\kappa\in\mathbb{R}
\]
explicitly emphasizing that arc length, area, and integral means are recovered from such
integrals by suitable choices of the parameters. In recent years, Thomas \cite{Tho16} revisited logarithmic-coefficient problems for close-to-convex functions, and Thomas, Tuneski, and Vasudevarao \cite{TTV18} published a modern monograph on univalent functions.

Integral means for functions with positive real part were already studied by Hayman \cite{Hay61} and by Holland--Twomey \cite{HT80}. Thomas asked the following question about the  growth law over the whole Carath\'eodory class, rather than for a fixed function or for a class determined by a prescribed zero sequence. This question was later recorded in Hayman and Lingham's well-known problem collection \emph{Research Problems in Function Theory} \cite[p.~182, Problem~6.123]{HL19}.

\begin{problem}[Thomas]\label{prob:HL}
Let \(\Pclass\) be the non-normalized Carath\'eodory class. For \(p\in\Pclass\), is it true that
\[
\int_0^{2\pi}\left|\frac{z p'(z)}{p(z)}\right|^2\,d\theta
=O\left(\frac{1}{1-r}\right),
\qquad z=re^{i\theta},\quad r\to1^-?
\]
If not, what is the correct rate of growth?
\end{problem}

For later use, we write
\[
I_p(r):=\int_0^{2\pi}\left|\frac{z p'(z)}{p(z)}\right|^2\,d\theta,
\qquad z=re^{i\theta},\quad 0<r<1.
\]
Problem~\ref{prob:HL} concerns the full Carath\'eodory class
and asks for a universal growth law for the family
\[
\{I_p:p\in\Pclass\}.
\]

The M\"obius map
\[
p_0(z):=\frac{1+z}{1-z}
\]
shows that the exponent \(1\) in Problem~\ref{prob:HL} is not accidental.
Indeed,
\[
\frac{z p_0'(z)}{p_0(z)}=\frac{2z}{1-z^2},
\]
and hence
\[
I_{p_0}(r)=4r^2\int_0^{2\pi}\frac{d\theta}{|1-r^2e^{2i\theta}|^2}
=\frac{8\pi r^2}{1-r^4}
\asymp \frac{1}{1-r}.
\]
Thus \((1-r)^{-1}\) growth does occur inside the class \(\Pclass\).

Our results show, however, that this behavior is not universal. The key observation is that the Carath\'eodory condition places the image of \(p\) in the right half-plane, so that \(\Log p\) has bounded imaginary part. This yields a uniform \(H^2\)-estimate for the nonconstant part of \(\Log p\), and hence a class-wide bound for \(\frac{z p'(z)}{p(z)}\) via Parseval's identity. In particular, we obtain a uniform \(O((1-r)^{-2})\) estimate over \(\Pclass\), and for each fixed function a standard tail argument yields the sharper little-\(o\) refinement. The optimality statement is established by explicit lacunary constructions showing that no smaller universal gauge can hold over the whole class.

We now state our main results. We begin with the universal upper bound and the
corresponding pointwise refinement for each fixed function.

\begin{theorem}\label{thm:upper}
For every \(p\in\Pclass\) and \(0<r<1\),
\[
I_p(r)\le \frac{\pi^3 e^{-2}}{(1-r)^2}.
\]
Moreover, for each fixed \(p\in\Pclass\),
\[
I_p(r)=o\bigl((1-r)^{-2}\bigr)
\qquad (r\to1^-).
\]
\end{theorem}

The next result shows that the exponent \(2\) is sharp on the power scale:
although every fixed function satisfies the little-\(o\) estimate above, no
uniform bound of the form \(O((1-r)^{-\beta})\) can hold over the whole class
when \(\beta<2\).

\begin{theorem}\label{thm:sharp}
There exists a function \(p_*\in\Pclass\) such that, for every real
\(\beta<2\),
\[
I_{p_*}(r)\neq O\bigl((1-r)^{-\beta}\bigr)
\qquad (r\to1^-).
\]
\end{theorem}

We then strengthen this sharpness statement from powers to arbitrary gauges.
The following theorem shows that \((1-r)^{-2}\) is optimal not only among power
functions, but among all universal comparison functions up to little-\(o\).

\begin{theorem}\label{thm:gauge}
Let \(\Phi:(0,1)\to(0,\infty)\) satisfy
\[
\Phi(r)=o\bigl((1-r)^{-2}\bigr)
\qquad (r\to1^-).
\]
Then there exists a function \(p_\Phi\in\Pclass\) such that
\[
I_{p_\Phi}(r)\neq O\bigl(\Phi(r)\bigr)
\qquad (r\to1^-).
\]
Equivalently,
\[
\limsup_{r\to1^-}\frac{I_{p_\Phi}(r)}{\Phi(r)}=+\infty.
\]
\end{theorem}

Combining these three results, we obtain a complete answer to
Problem~\ref{prob:HL} and a precise formulation of the optimal universal growth
scale for the family \(\{I_p:p\in\Pclass\}\).

\begin{corollary}\label{cor:main}
Problem~\ref{prob:HL} has a negative answer. Moreover, \((1-r)^{-2}\) is the
optimal universal growth scale for the family \(\{I_p:p\in\Pclass\}\) in the
following sense:
\begin{enumerate}[label=\textup{(\roman*)}]
\item there exists an absolute constant \(C>0\) such that
\[
I_p(r)\le C(1-r)^{-2},
\qquad p\in\Pclass,\quad 0<r<1;
\]
\item for each fixed \(p\in\Pclass\),
\[
I_p(r)=o\bigl((1-r)^{-2}\bigr)
\qquad (r\to1^-);
\]
\item if \(\Phi:(0,1)\to(0,\infty)\) satisfies
\[
\Phi(r)=o\bigl((1-r)^{-2}\bigr)
\qquad (r\to1^-),
\]
then there exists a function \(p_\Phi\in\Pclass\) such that
\[
I_{p_\Phi}(r)\neq O\bigl(\Phi(r)\bigr)
\qquad (r\to1^-);
\]
\item in particular, \(2\) is the least exponent \(\alpha\) for which a uniform estimate
\[
I_p(r)\le C_\alpha (1-r)^{-\alpha},
\qquad p\in\Pclass,\quad 0<r<1,
\]
can hold with a constant \(C_\alpha\) independent of \(p\).
\end{enumerate}
\end{corollary}

The proof of Corollary~\ref{cor:main} will be given after
Theorem~\ref{thm:gauge}.

\medskip
\noindent\textbf{Organization of this paper.}
Section~\ref{sec:h2} establishes an \(H^2\)-estimate for \(\Log p\).
In Section~\ref{sec:upper} we combine this estimate with Parseval's identity to
prove Theorem~\ref{thm:upper}. Section~\ref{sec:sharp} gives an explicit
lacunary construction proving Theorem~\ref{thm:sharp}. In
Section~\ref{sec:gauge} we prove Theorem~\ref{thm:gauge} and then derive
Corollary~\ref{cor:main}.

\section{An \texorpdfstring{$H^2$}{H2}-estimate for the nonconstant part of \texorpdfstring{$\Log p$}{Log p}}\label{sec:h2}

Fix \(p\in\Pclass\). Since \(p(\D)\) is contained in the right half-plane, the principal branch of the logarithm is analytic there. We may therefore write
\[
F(z):=\Log p(z)=a_0+\sum_{n=1}^{\infty} a_n z^n,
\qquad z\in\D.
\]
Then
\[
F'(z)=\frac{p'(z)}{p(z)}
\qquad\text{and}\qquad
|\operatorname{Im} F(z)|<\frac{\pi}{2}
\quad (z\in\D).
\]
The bounded imaginary part of \(F=\Log p\) yields the following coefficient estimate. Recall that \(H^2\) denotes the classical Hardy space on \(\D\), consisting of those functions
\[
f(z)=\sum_{n=0}^{\infty} b_n z^n\in H(\D)
\]
for which
\[
\|f\|_{H^2}^2
:=\sup_{0<r<1}\frac{1}{2\pi}\int_0^{2\pi}|f(re^{i\theta})|^2\,d\theta
<\infty.
\]
Equivalently,
\[
f\in H^2 \quad\Longleftrightarrow\quad \sum_{n=0}^{\infty}|b_n|^2<\infty,
\]
and in that case one has \(\|f\|_{H^2}^2=\sum_{n=0}^{\infty}|b_n|^2\).

\begin{lemma}\label{lem:H2}
With the notation above,
\[
\sum_{n=1}^{\infty}|a_n|^2\le \frac{\pi^2}{2}.
\]
In particular, \(F\in H^2\).
\end{lemma}

\begin{proof}
For \(0<r<1\), write \(z=re^{i\theta}\). Since
\[
F(z)=a_0+\sum_{n=1}^{\infty} a_n r^n e^{in\theta},
\]
we obtain
\[
\operatorname{Im} F(re^{i\theta})
=\operatorname{Im} a_0+
\frac{1}{2i}\sum_{n=1}^{\infty} a_n r^n e^{in\theta}
-\frac{1}{2i}\sum_{n=1}^{\infty} \overline{a_n}\,r^n e^{-in\theta}.
\]
Applying Parseval's identity to this Fourier series yields
\[
\frac{1}{2\pi}\int_0^{2\pi}|\operatorname{Im} F(re^{i\theta})|^2\,d\theta
=(\operatorname{Im} a_0)^2+\frac12\sum_{n=1}^{\infty}|a_n|^2 r^{2n}.
\]
On the other hand, \(|\operatorname{Im} F(re^{i\theta})|<\pi/2\), so
\[
\frac{1}{2\pi}\int_0^{2\pi}|\operatorname{Im} F(re^{i\theta})|^2\,d\theta
\le \frac{\pi^2}{4}.
\]
Combining the last two identities gives
\[
(\operatorname{Im} a_0)^2+\frac12\sum_{n=1}^{\infty}|a_n|^2 r^{2n}\le \frac{\pi^2}{4}
\qquad (0<r<1).
\]
Letting \(r\to1^-\) and using monotone convergence, we conclude that
\[
\sum_{n=1}^{\infty}|a_n|^2\le \frac{\pi^2}{2}.
\]
This proves the lemma.
\end{proof}

The lemma provides the coefficient estimate needed for the proof of
Theorem~\ref{thm:upper}. We now turn to the integral means \(I_p(r)\).

\section{A uniform upper bound and a fixed-function refinement}\label{sec:upper}
Fix \(p\in\Pclass\), and write
\[
F(z)=\Log p(z)=a_0+\sum_{n=1}^{\infty} a_n z^n,
\qquad z\in\D.
\]
We now return to the quantity appearing in Problem~\ref{prob:HL}. In terms of \(F=\Log p\), Parseval's identity gives an exact series representation for \(I_p(r)\).

\begin{proposition}\label{prop:parseval}
For every \(0<r<1\),
\[
I_p(r)=2\pi\sum_{n=1}^{\infty} n^2 |a_n|^2 r^{2n}.
\]
\end{proposition}

\begin{proof}
Since
\[
\frac{z p'(z)}{p(z)}=zF'(z)=\sum_{n=1}^{\infty} n a_n z^n,
\]
Parseval's identity gives
\[
I_p(r)=\int_0^{2\pi}\left|\sum_{n=1}^{\infty} n a_n r^n e^{in\theta}\right|^2\, d\theta
=2\pi\sum_{n=1}^{\infty} n^2 |a_n|^2 r^{2n}.
\qedhere\]
\end{proof}

Thus the proof of Theorem~\ref{thm:upper} reduces to estimating the weighted series on the right-hand side. Combined with Lemma~\ref{lem:H2}, this yields a uniform \(O((1-r)^{-2})\) bound over \(\Pclass\), and a standard tail argument then gives the little-\(o\) refinement for each fixed function.


\begin{proof}[Proof of Theorem~\ref{thm:upper}]
By Proposition~\ref{prop:parseval},
\begin{equation}\label{eq:Ip-series}
I_p(r)=2\pi\sum_{n=1}^{\infty} n^2 |a_n|^2 r^{2n}.
\end{equation}
Since
\[
r^n\le e^{-n(1-r)} \qquad (0<r<1),
\]
we have
\[
(1-r)^2 n^2 r^{2n}
\le \bigl(n(1-r)\bigr)^2 e^{-2n(1-r)}
\le \sup_{x\ge 0} x^2 e^{-2x}
= e^{-2}.
\]
Therefore, by \eqref{eq:Ip-series} and Lemma~\ref{lem:H2},
\[
I_p(r)
\le \frac{2\pi e^{-2}}{(1-r)^2}\sum_{n=1}^{\infty}|a_n|^2
\le \frac{\pi^3 e^{-2}}{(1-r)^2},
\qquad p\in\Pclass,\quad 0<r<1.
\]
This proves the stated uniform estimate.

To obtain the little-\(o\) refinement for a fixed \(p\), let \(\varepsilon>0\).
By Lemma~\ref{lem:H2}, we can choose \(N\) so large that
\[
\sum_{n>N}|a_n|^2<\varepsilon.
\]
Splitting the series in \eqref{eq:Ip-series}, we get
\[
(1-r)^2 I_p(r)
=2\pi (1-r)^2\sum_{n\le N} n^2|a_n|^2 r^{2n}
+2\pi (1-r)^2\sum_{n>N} n^2|a_n|^2 r^{2n}.
\]
The first term tends to \(0\) as \(r\to1^-\), because the sum is finite. For the tail we use the estimate above:
\[
(1-r)^2\sum_{n>N} n^2|a_n|^2 r^{2n}
\le e^{-2}\sum_{n>N}|a_n|^2
<e^{-2}\varepsilon.
\]
Hence
\[
\limsup_{r\to1^-}(1-r)^2 I_p(r)\le 2\pi e^{-2}\varepsilon.
\]
Since \(\varepsilon>0\) is arbitrary, it follows that
\[
I_p(r)=o\bigl((1-r)^{-2}\bigr)
\qquad (r\to1^-).
\qedhere\]
\end{proof}

This settles the upper-bound side of Problem~\ref{prob:HL}: the family
\(\{I_p:p\in\Pclass\}\) satisfies the uniform estimate
\(I_p(r)=O\bigl((1-r)^{-2}\bigr)\), while each fixed function enjoys the sharper
little-\(o\) refinement. The next section shows that the exponent \(2\) cannot
be improved on the power scale.

\section{Sharpness on the power scale}\label{sec:sharp}

\begin{proof}[Proof of Theorem~\ref{thm:sharp}]
Define
\[
F_*(z):=\frac{i}{2}\sum_{k=1}^{\infty}\frac{z^{2^k}}{k^2},
\qquad
p_*(z):=e^{F_*(z)}.
\]
Since \(\sum_{k\ge1}k^{-2}<\infty\), the series for \(F_*\) converges absolutely and uniformly on \(\overline{\D}\). Moreover,
\[
|\operatorname{Im} F_*(z)|
\le \frac12\sum_{k=1}^{\infty}\frac{|z|^{2^k}}{k^2}
\le \frac12\sum_{k=1}^{\infty}\frac{1}{k^2}
=\frac{\pi^2}{12}<\frac{\pi}{2}.
\]
Therefore
\[
\operatorname{Re} p_*(z)=e^{\operatorname{Re} F_*(z)}\cos(\operatorname{Im} F_*(z))>0,
\qquad z\in\D,
\]
so \(p_*\in\Pclass\).

Also,
\[
\frac{z p_*'(z)}{p_*(z)}=zF_*'(z)
=\frac{i}{2}\sum_{k=1}^{\infty}\frac{2^k}{k^2}z^{2^k}.
\]
Hence Parseval's identity yields
\begin{equation}\label{eq:counterexample-series}
I_{p_*}(r)=2\pi\sum_{k=1}^{\infty}\frac{4^{k-1}}{k^4}r^{2^{k+1}}.
\end{equation}

To extract the growth from \eqref{eq:counterexample-series}, we choose radii adapted to the exponents:
\[
r_k:=e^{-2^{-k}}, \qquad k\ge1.
\]
Then
\[
r_k^{2^{k+1}}=e^{-2}.
\]
Moreover, since \(1-e^{-x}\le x\) for \(x\ge0\) and \(1-e^{-x}\ge x/2\) for \(0<x\le1\), we have
\[
2^{-k-1}\le 1-r_k\le 2^{-k}
\qquad (k\ge1).
\]
Fix \(\beta<2\). From these inequalities it follows that there exists a constant \(c_\beta>0\) such that
\[
(1-r_k)^\beta\ge c_\beta\,2^{-k\beta}
\qquad (k\ge1).
\]
Now keep only the \(k\)th term in \eqref{eq:counterexample-series}. This gives
\[
I_{p_*}(r_k)
\ge 2\pi\frac{4^{k-1}}{k^4}r_k^{2^{k+1}}
=2\pi e^{-2}\frac{4^{k-1}}{k^4}.
\]
Consequently,
\[
I_{p_*}(r_k)(1-r_k)^\beta
\ge 2\pi e^{-2}\frac{4^{k-1}}{k^4}\,c_\beta 2^{-k\beta}
= C_\beta\,\frac{2^{(2-\beta)k}}{k^4},
\]
where \(C_\beta>0\) is independent of \(k\). Since \(2-\beta>0\), the right-hand side tends to \(+\infty\) as \(k\to\infty\). Therefore
\[
I_{p_*}(r)\neq O\bigl((1-r)^{-\beta}\bigr)
\qquad (r\to1^-)
\]
for every real \(\beta<2\).
\end{proof}

The function \(p_*\) shows that the exponent \(2\) is optimal on the power
scale. We now strengthen this by proving that no universal gauge
\(o((1-r)^{-2})\) can hold on the whole class \(\Pclass\).

\section{Optimality of the universal gauge}\label{sec:gauge}

\begin{proof}[Proof of Theorem~\ref{thm:gauge}]
For \(n\ge1\), set
\[
\rho_n:=e^{-1/n}.
\]
Since \(1-\rho_n=1-e^{-1/n}\sim 1/n\), the hypothesis
\[
\Phi(r)=o\bigl((1-r)^{-2}\bigr)
\qquad (r\to1^-)
\]
implies
\[
\frac{\Phi(\rho_n)}{n^2}\longrightarrow 0
\qquad (n\to\infty).
\]
Hence we may choose inductively a strictly increasing sequence of integers \(n_k\to\infty\) such that
\begin{equation}\label{eq:choose-nk}
\Phi\bigl(e^{-1/n_k}\bigr)\le \frac{n_k^2}{k^8}
\qquad (k\ge1).
\end{equation}
Define
\[
F_\Phi(z):=\frac{i}{2}\sum_{k=1}^{\infty}\frac{z^{n_k}}{k^2},
\qquad
p_\Phi(z):=e^{F_\Phi(z)}.
\]
Again the series converges absolutely and uniformly on \(\overline{\D}\), and
\[
|\operatorname{Im} F_\Phi(z)|
\le \frac12\sum_{k=1}^{\infty}\frac{|z|^{n_k}}{k^2}
\le \frac12\sum_{k=1}^{\infty}\frac1{k^2}
=\frac{\pi^2}{12}<\frac{\pi}{2}.
\]
Therefore
\[
\operatorname{Re} p_\Phi(z)=e^{\operatorname{Re} F_\Phi(z)}\cos(\operatorname{Im} F_\Phi(z))>0,
\qquad z\in\D,
\]
so \(p_\Phi\in\Pclass\).

Since \(F_\Phi\) is a power series with radius of convergence at least \(1\), termwise differentiation is valid in \(\D\), and hence
\[
\frac{z p_\Phi'(z)}{p_\Phi(z)}=zF_\Phi'(z)
=\frac{i}{2}\sum_{k=1}^{\infty}\frac{n_k}{k^2}z^{n_k}.
\]
Therefore Parseval's identity gives
\begin{equation}\label{eq:gauge-series}
I_{p_\Phi}(r)=\frac{\pi}{2}\sum_{k=1}^{\infty}\frac{n_k^2}{k^4}r^{2n_k}.
\end{equation}
Now put
\[
r_k:=e^{-1/n_k}.
\]
From \eqref{eq:gauge-series}, keeping only the \(k\)th term, we obtain
\[
I_{p_\Phi}(r_k)
\ge \frac{\pi}{2}\frac{n_k^2}{k^4}r_k^{2n_k}
=\frac{\pi e^{-2}}{2}\frac{n_k^2}{k^4}.
\]
Combining this with \eqref{eq:choose-nk}, we find
\[
\frac{I_{p_\Phi}(r_k)}{\Phi(r_k)}
\ge \frac{\pi e^{-2}}{2}\,
\frac{n_k^2/k^4}{n_k^2/k^8}
=\frac{\pi e^{-2}}{2}k^4.
\]
Hence
\[
\limsup_{k\to\infty}\frac{I_{p_\Phi}(r_k)}{\Phi(r_k)}=+\infty,
\]
and therefore
\[
\limsup_{r\to1^-}\frac{I_{p_\Phi}(r)}{\Phi(r)}=+\infty.
\]
This proves Theorem~\ref{thm:gauge}.
\end{proof}

\begin{proof}[Proof of Corollary~\ref{cor:main}]
Part~\textup{(i)} is the uniform estimate from Theorem~\ref{thm:upper}, and
part~\textup{(ii)} is its little-\(o\) refinement for fixed functions.
Part~\textup{(iii)} is exactly Theorem~\ref{thm:gauge}. Finally,
part~\textup{(iv)} follows from part~\textup{(i)} with \(\alpha=2\), while
part~\textup{(iii)} applied to \(\Phi(r)=(1-r)^{-\alpha}\) shows that no
exponent \(\alpha<2\) can work.
\end{proof}

\begin{remark}
Corollary~\ref{cor:main} concerns universal bounds for the family
\(\{I_p:p\in\Pclass\}\), rather than a single common growth law for all
functions in \(\Pclass\). Indeed, for the M\"obius map \(p_0(z)=(1+z)/(1-z)\), one has 
\[I_{p_0}(r)\asymp (1-r)^{-1},\]
whereas \(I_{p_*}(r)\) from Theorem~\ref{thm:sharp} grows, along a suitable sequence,
faster than \((1-r)^{-\beta}\) for every \(\beta<2\). Thus different functions in \(\Pclass\) may exhibit substantially different
growth behavior near the boundary, even though \((1-r)^{-2}\) is the optimal
universal scale for the class as a whole.
\end{remark}

\section*{Acknowledgments}
Teng Zhang is supported by the China Scholarship Council, the Young Elite Scientists Sponsorship Program for PhD Students (China Association for Science and Technology), and the Fundamental Research Funds for the Central Universities at Xi'an Jiaotong University (Grant No.~xzy022024045).

\end{document}